\documentclass[11pt,reqno]{amsart}

\usepackage{enumerate}
\newtheorem{theorem}{Theorem}[section]

\theoremstyle{remark}
\newtheorem{remark}{Remark}[section]

\numberwithin{equation}{section}
 \DeclareMathOperator\hdim{\dim_H}

\def\N{\mathbb{N}}

\begin{document}

\title[On exceptional sets in  Erd\H{o}s-R\'{e}nyi limit theorem revisited]{On exceptional sets in Erd\H{o}s-R\'{e}nyi limit theorem revisited}

\author[Jinjun Li]{Jinjun Li*}

 \address[Jinjun Li]{Department of Mathematics,~Minnan Normal
University,~Zhang\-zhou, 363000, P.~R.~China}
\email{li-jinjun@163.com}

\author{Min Wu}
\address[Min Wu]{Department of Mathematics, South China
 University of Technology, Guang\-zhou, 510641, P.~R.~China}
\email{wumin@scut.edu.cn}

\thanks{*Corresponding author. Email: li-jinjun@163.com}

\subjclass{Primary 28A80; 60F99; Secondary 54E52}

\keywords{run-length function, Erd\H{o}s-R\'{e}nyi limit theorem, Hausdorff dimension, residual set.}

%\date{January 1, 2004}
%----------additions
%\dedicatory{To my boss}
%%% ----------------------------------------------------------------------

\begin{abstract}

For $x\in [0,1],$ the run-length function $r_n(x)$ is defined as the length of the longest run of $1$'s amongst the first $n$ dyadic digits in the dyadic expansion of $x.$  Let $H$ denote the set of monotonically increasing functions $\varphi:\mathbb{N}\to (0,+\infty)$ with $\lim\limits_{n\to\infty}\varphi(n)=+\infty$. For any $\varphi\in H$, we prove that the set
\[
E_{\max}^\varphi=\left\{x\in [0,1]:\liminf\limits_{n\to\infty}\frac{r_n(x)}{\varphi(n)}=0, \limsup\limits_{n\to\infty}\frac{r_n(x)}{\varphi(n)}=+\infty\right\}
\]
either has Hausdorff dimension one and is residual in $[0,1]$ or is empty. The result solves a conjecture posed in \cite{LW5} affirmatively.
\end{abstract}

%%% ----------------------------------------------------------------------
\maketitle
%%% ----------------------------------------------------------------------
%\tableofcontents
\section{Introduction}
The aim of the present paper is to solve a problem on exceptional sets in Erd\H{o}s-R\'{e}nyi limit theorem which was posed in \cite{LW5}. Let us recall the background and some notation in \cite{LW5}. The run-length function $r_n(x),$ which was introduced to measure the length of consecutive terms of ``heads" in $n$ Bernoulli trials, is defined as follows. It is well known that every $x\in [0,1]$ admits a dyadic expansion:
\[
x=\sum_{k=1}^\infty\frac{x_k}{2^k},
\]
where $x_k\in \{0,1\}$ for any $k\ge 1$. Write $\Sigma^\infty=\{0,1\}^\N.$ The infinite sequence $(x_1,x_2, x_3,\ldots)\in \Sigma^\infty$ is called the digit sequence of $x.$ Let $\pi:\Sigma^\infty\to [0,1]$ be the code map, that is, $\pi ((x_1,x_2, x_3,\ldots))=x.$ For each $n\ge 1$ and $x\in [0,1]$, the run-length function $r_n(x)$ is defined as the length of the longest run of $1$'s in $(x_1,x_2,\ldots,x_n)$, that is,
\[
r_n(x)=\max\{\ell: \text{$x_{i+1}=\cdots=x_{i+\ell}=1$ for some $0\le i\le n-\ell$}\}.
\]
The run-length function has been extensively studied in probability theory and used in reliability theory, biology, quality control.
 Erd\H{o}s and R\'{e}nyi \cite{ER} (see also \cite{Re}) proved the following asymptotic behavior of $r_n$: for Lebesgue almost all $x\in[0,1]$,
\begin{equation}\label{lim}
\lim\limits_{n\to\infty}\frac{r_n(x)}{\log_2n}=1.
\end{equation}

Roughly speaking, the rate of growth of $r_n(x)$ is $\log_2n$ for almost all $x\in [0,1].$  Recently, some special sets consisting of points whose run-length function obey other asymptotic behavior instead of $\log_2n$ was considered by Zou \cite{Zou}. Chen and Wen \cite{CW} studied some level sets on the frequency involved in dyadic expansion and run-length function.  For more details about the run-length function, we refer the reader to the book \cite{Re}.

The limit in \eqref{lim} may not exist. Therefore, it is natural to study the exceptional set in the above Erd\H{o}s-R\'{e}nyi limit theorem.
Let
\[
E=\left\{x\in [0,1]:\liminf\limits_{n\to\infty}\frac{r_n(x)}{\log_2n}<\limsup\limits_{n\to\infty}\frac{r_n(x)}{\log_2n}\right\}.
\]
It follows from the Erd\H{o}s-R\'{e}nyi limit theorem that the set $E$ is negligible from the measure-theoretical point of view. On the other hand, we also often employ some fractal dimensions to characterize the size of a set. Hausdorff dimension perhaps is the most popular one.  Ma et al. \cite{MWW} proved that the set of points that violate the above Erd\H{o}s and R\'{e}nyi law has full Hausdorff dimension. It is worth to point out that $E$ is smaller than the set considered in \cite{MWW} because we consider the asymptotic behavior of $r_n(x)$ with respect to the fixed speed $\log_2 n.$ There is a natural question: what is the Hausdorff dimesnion of the set $E$? In fact, questions related to the exceptional sets  from dynamics and fractals have recently attracted huge interest in the literature.
Generally speaking, exceptional sets are big from the dimensional point of view, and they have the same fractal dimensions as the underlying phase spaces, see \cite{APT,APT2,BS,FJW,FW,KL,LWX,LW4,OW,Pe,PP,PT} and references therein.

Define
\[
E_{\max}=\left\{x\in [0,1]:\liminf\limits_{n\to\infty}\frac{r_n(x)}{\log_2n}=0, \limsup\limits_{n\to\infty}\frac{r_n(x)}{\log_2n}=+\infty\right\}.
\]
That is,  $E_{\max}$ is the set consisting of those ``worst'' divergence points.
Clearly,  $E_{\max}\subset E.$

Intuitively, we feel that the set $E_{\max}$ shall be ``small''. However, the authors showed in \cite{LW5} that $\hdim E_{\max}=1.$ Here and in the sequel, $\hdim E$ denotes the Hausdorff dimension of the set $E$.  For more details about Hausdorff dimension and the theory of fractal dimensions, we refer the reader to the famous book \cite{Fal}.

Moreover, it is also natural to study the asymptotic behavior of run-length function with respect to other speeds instead of $\log_2n.$  In fact, In \cite{LW5} the authors proved that the exceptional sets with respect to a more general class of speeds still have full dimensions. To state the result, we need to introduce some notation. Let $H$ denote the set of monotonically increasing function $\varphi:\mathbb{N}\to (0,+\infty)$ with $\lim\limits_{n\to +\infty}f(n)=+\infty$.
For $\varphi\in H$, define

\begin{equation}\label{jhf}
E_{\max}^\varphi=\left\{x\in [0,1]:\liminf\limits_{n\to\infty}\frac{r_n(x)}{\varphi(n)}=0, \limsup\limits_{n\to\infty}\frac{r_n(x)}{\varphi(n)}=+\infty\right\}.
\end{equation}
Consider the following subset of $H$:
 \[
 A=\left\{\varphi\in H: \text{there exists $0<\alpha\le 1$ such that $\limsup\limits_{n\to\infty}\frac{n}{\varphi(n^{1+\alpha})}=+\infty$}\right\}.
 \]
In \cite{LW5}, the authors obtained the following result.
\begin{theorem}[\cite{LW5}]\label{MR11}
Let $\varphi\in A$ and $E_{\max}^\varphi$ be defined as in $\eqref{jhf}$. Then
\[
\hdim E_{\max}^\varphi=1.
\]
\end{theorem}

It is easy to check that $\log_2 n\in A$ and $n^\beta\in A$, where $0<\beta<1$.
Unfortunately, many other speeds are not in the set $A.$ In \cite{LW5}, the authors conjectured that for $\varphi\in H$, the Hausdorff dimension of $E_{\max}^\varphi$ is either one or zero. In this paper, we solve this conjecture affirmatively. More precisely, we have the following result.

\begin{theorem}\label{MRR1}
Let $\varphi\in H$ and $E_{\max}^\varphi$ be defined as in $\eqref{jhf}$.
\begin{enumerate}
\item [\textup{(1)}] If $\limsup\limits_{n\to \infty}\frac{n}{\varphi(n)}=+\infty$ then $\hdim E_{\max}^\varphi=1;$
\item [\textup{(2)}] If $\limsup\limits_{n\to \infty}\frac{n}{\varphi(n)}<+\infty$ then $\hdim E_{\max}^\varphi=0.$
\end{enumerate}
\end{theorem}
\begin{remark}\label{R1}
In fact, under the condition $\limsup\limits_{n\to \infty}\frac{n}{\varphi(n)}<+\infty$ the set $E_{\max}^\varphi=\emptyset$ since $r_n(x)\le n$ for any $n\ge 1$ and $x\in [0,1].$
\end{remark}
We would like to emphasize that the method in \cite{LW5} cannot be applied to the generalized functions $\varphi\in H$. To prove Theorem \ref{MRR1}, we will use another method which is inspired by the idea in \cite{FW}.

Clearly,  if $\varphi\in A$ then $\varphi$ satisfies the condition in the first part of Theorem \ref{MRR1}. Therefore, Theorem \ref{MRR1} generalizes Theorem \ref{MR11}.

We can also discuss the size of $E_{\max}^\varphi$ from the topological point of view, which is another way to describe the size of a set.
 Recall that in a metric space $X$,
 a set $R$ is called residual if its complement is of the first category. Moreover, in a complete metric space a set
 is residual if it contains a dense $G_\delta$ set, see \cite{Ox}.  Recent results show that certain exceptional sets
 can also be large from the topological point of view, see, for example, \cite{BLV,BLV2,BLV2,HLOPS,LW2,LW3,LWW,LW4,Ol2} and references therein. In \cite{LW5} the authors also showed that the set $E_{\max}^\varphi$ is residual if $\varphi\in A$. The following result is a generalization of it.

\begin{theorem}\label{MRR2}
Let $\varphi\in H$ with $\limsup\limits_{n\to \infty}\frac{n}{\varphi(n)}=+\infty$ and $E_{\max}^\varphi$ be defined as in \eqref{jhf}. Then the set $E_{\max}^\varphi$ is residual in $[0,1].$
\end{theorem}
Noting Remark \ref{R1}, Theorem \ref{MRR2} tells us that for any $\varphi\in H$ the set $E_{\max}^\varphi$ is either residual or empty.

\section{Proofs }
This section is devoted to the proofs of Theorem $\ref{MRR1}$ and Theorem \ref{MRR2}. First we need to introduce some notation.
For $n\in \N$, let
\[
\Sigma^{n}=\{(x_1,\cdots,x_n):x_i\in\{0,1\},i=1,\cdots,n\}
\] and
\[
\Sigma^{*}=\bigcup_{n \in \N}\Sigma^{n}.
\]
For each $\omega=(\omega_1,\cdots,\omega_n)\in \Sigma^{n}$, let $|\omega|=n$ denote the length of the word $\omega.$ For $\omega=(\omega_{1}\ldots \omega_{n})\in \Sigma^{n}$ and a positive integer $m$ with $m\le n$,
or for $\omega=(\omega_{1},\omega_{2},\ldots )\in \Sigma^\infty$ and a positive integer $m$, let $\omega|_m=(\omega_{1}\ldots \omega_{m})$ denote the truncation of $\omega$ to the $m$th place.
For two words $\omega=(\omega_1,\omega_2,\cdots, \omega_n)\in\Sigma^n$ and $\tau=(\tau_1,\tau_2,\cdots,\tau_m)\in\Sigma^m$, we denote their concatenation by $\omega\tau=(\omega_1,\cdots,\omega_n,\tau_1,\cdots,\tau_m),$ which is a word of length $n+m.$

It is well known that the space $\Sigma^\infty$ is
compact if it is equipped with the usual metric defined by
\[
d(x,y)=2^{-\min\{k:x_{k+1}\not=y_{k+1}\}}, x,y\in\Sigma^\infty.
\]
In this paper, the notation $[x]$ denotes the integer part of $x$, and the notation $a^m $ means $\underbrace{a\cdots a}_{\text{$m$ times}}$, where $a=0$ or $1.$
\begin{proof}[Proof of Theorem $\ref{MRR1}$.]

Noting Remark \ref{R1}, we only need to prove the first part. Let $\varphi\in H$ and $\limsup\limits_{n\to \infty}\frac{n}{\varphi(n)}=+\infty$.
It is sufficient to show that $\hdim E_{\max}^\varphi \ge 1-\gamma$ for any $\gamma>0$.

Fix $0<\gamma<1$. Choose $p\in \mathbb{N}$ such that $\frac{p-2}{p}>1-\gamma.$ Let
\[
E_p=\{x\in\Sigma^\infty: \text{$x_i=1$ for $1\le i\le p$ and $x_{kp+1}=x_{kp+p}=0$ for any $k\ge 1$}\}.
\]
Since the set $E_p$ can be viewed as a self-similar set generated by $2^{p-2}$ similitudes with ratio $2^{-p}$, it follows from the dimensional formula of self-similar set (see, for example, \cite{Fal}) that $\hdim E_p=\frac{p-2}{p}.$

In the follows, by means of $E_p$ we will construct a set $E^*_p$ such that $\pi(E^*_p)\subset E_{\max}^\varphi$ and define a one-to-one map $f$ from $E_p$ onto $E^*_p$. Moreover, the map $f^{-1}$ is nearly Lipschitz on $f(E_p)$, namely,  for any $\varepsilon>0$, there exists some $N_0$ such that $d(f(x),f(y))<2^{-n}$ implies that $d(x,y)<2^{-n(1-\varepsilon)}$ for any $n>N_0.$ This implies that $\hdim E^*_p=\hdim f(E_p)\ge \hdim E_p$ (see Proposition 2.3 of \cite{Fal}) . Therefore,
\[
\hdim E_{\max}^\varphi\ge \hdim \pi(E^*_p)=\hdim f(E_p)\ge \hdim E_p=\frac{p-2}{p}>1-\gamma,
\]
as desired. Here the first equality follows from that facts that $\pi:\Sigma^\infty\to [0,1]$ is a bi-Lipschitz map except a countable set (the set of  binary endpoints) and the Hausdorff dimension of a countable set is zero.

Next we will construct the desired set $E^*_p.$  Let $n_0=\inf\{n:\log \varphi(n)\ge p\}$.

Since $\limsup\limits_{n\to \infty}\frac{n}{\varphi(n)}=+\infty$, there exists a subsequence $\{n_k\}_{k\ge 1}$ such that
\begin{equation}\label{con1}
n_k \ge 2^{n_{k-1}}, k\ge 1
\end{equation}
and
\begin{equation}\label{con2}
\lim\limits_{k\to \infty}\frac{n_k}{\varphi(n_k)}=+\infty.
\end{equation}

For each $x=(x_i)_{i\ge1}\in E_p,$ we will construct a sequence $\{x^{(n)}\}_{n\ge 0}$ of points in $\Sigma^\infty$ by induction. Let $x^{(0)}=x.$ Suppose we have defined $x^{(j)}=x_1^{(j)}x_2^{(j)}\cdots x_n^{(j)}\cdots$ for $0\le j\le 2k$. We define $x^{(2k+1)}$ and $x^{(2k+2)}$ as follows.
Let
\[
x^{(2k+1)}=x_{1}^{(2k)}\cdots x^{(2k)}_{n_{2k+1}-1-[\log \varphi(n_{2k+1})]}01^{[\log\varphi(n_{2k+1})]}0 x_{n_{2k+1}-[\log \varphi(n_{2k+1})]}^{(2k)}\cdots.
\]
 Namely, $x^{(2k+1)}$ is obtained by inserting word $01^{[\log\varphi(n_{2k+1})]}0$ in $x^{(2k)}$ at the place $n_{2k+1}-[\log \varphi(n_{2k+1})]$. Then, let \[
 x^{(2k+2)}=x^{(2k+1)}_1\cdots x^{(2k+1)}_{n_{2k+2}-1-[\frac{1}{2}n_{2k+2}]}01^{[\frac{1}{2}n_{2k+2}]}0 x_{n_{2k+2}-[\frac{1}{2}n_{2k+2}]}^{(2k+1)}\cdots.
 \]
 Similarly, $x^{(2k+2)}$ is obtained by inserting word $01^{[\frac{1}{2}n_{2k+2}]}0$ in $x^{(2k+1)}$ at the place $n_{2k+2}-[\frac{1}{2}n_{2k+2}]$.

By \eqref{con1}, we have $n_{2k+2}-1-[\frac{1}{2}n_{2k+2}]>n_{2k+1}+1$. Therefore, it is not difficult to check that
\[
x^{(2k+2)}|_{n_{2k+2}-1-[\frac{1}{2}n_{2k+2}]}=x^{(2k+1)}|_{n_{2k+2}-1-[\frac{1}{2}n_{2k+2}]}.
\]
 In other words, for $k\ge 1$ we have
\[
d(x^{(2k+1)},x^{(2k+2)})\le \left(\frac{1}{2}\right)^{\left(n_{2k+2}-1-[\frac{1}{2}n_{2k+2}]\right)},
\]
which implies that the limit of $\{x^{(n)}\}_{n\ge 0}$ exists. Let $ \lim\limits_{n\to \infty}x^{(n)}=x^*.$

 We claim that $\pi(x^*)\in E_{\max}^\varphi$. In fact, it follows from the construction of $x^*$ that $r_{n_{2k-1}}(\pi(x^*))=[\log \varphi(n_{2k-1})]$ and $r_{n_{2k}}(\pi(x^*))=[\frac{1}{2}n_{2k}]$ for $k\ge 1.$ Therefore, by \eqref{con1} and \eqref{con2} we have
 \[
 \liminf\limits_{n\to \infty}\frac{r_{n}(\pi(x^*))}{\varphi(n)}\le\lim\limits_{k\to \infty}\frac{r_{n_{2k-1}}(\pi(x^*))}{\varphi(n_{2k-1})}=\lim\limits_{k\to \infty}\frac{[\log \varphi(n_{2k-1})]}{\varphi(n_{2k-1})}=0
 \]
and
\[
\limsup\limits_{n\to \infty}\frac{r_{n}(\pi(x^*))}{\varphi(n)}\ge \lim\limits_{k\to \infty}\frac{r_{n_{2k}}(\pi(x^*))}{\varphi(n_{2k})}=\lim\limits_{k\to\infty}\frac{[\frac{1}{2}n_{2k}]}{\varphi(n_{2k})}=+\infty.
\]

Define map $f:E_p\to \Sigma^\infty$ by $f(x)=x^*$ and let $E^*_p=f(E_p)$.

Clearly, $f$ is injective. We now show that the map $f^{-1}$ is nearly Lipschitz on $E^*_p$.  For any $n\ge n_0$, there exists some $k\ge 1$ such that $n_{2k-1}\le n< n_{2k}$ or $n_{2k}\le n< n_{2k+1}$. We only discuss the case $n_{2k-1}\le n< n_{2k}$ because the other case can be treated similarly. Suppose $d(f(x),f(y))=d(x^*,y^*)<2^{-n}$, then $x^*_1x^*_2\cdots x^*_n=y^*_1y^*_2\cdots y^*_n$. Since $x=f^{-1}(x^*), y=f^{-1}(y^*)$ is obtained by removing the inserted words we have $x_1x_2\cdots x_{n'}=y_1y_2\cdots y_{n'}$, where
\[
n'=n-\left((2+[\log \varphi(n_1)])+(2+[\frac{1}{2}n_2])\ldots+(2+[\log \varphi(n_{2k-1})])\right).
\]
By \eqref{con1} and the fact that $\lim\limits_{n\to \infty}\varphi(n)=+\infty$, we have
\[
\lim\limits_{k\to \infty}\frac{\left(2+[\log \varphi(n_1)])+(2+[\frac{1}{2}n_2])\ldots+(2+[\log \varphi(n_{2k-1})]\right)}{n_{2k-1}}=0.
\]
Therefore, for any $\varepsilon>0$ there exists an integer $N_1>n_0$ such that $n'> n-\varepsilon n_{2k-1}\ge (1-\varepsilon)n$ for any $n>N_1$
Therefore, $d(x,y)<2^{-n(1-\varepsilon)}$.
Similarly, in the other case there exists some integer $N_2>n_0$ such that $d(x,y)<2^{-n(1-\varepsilon)}$ for any $n>N_2$. Finally, let $N_0=\max\{N_1,N_2\}$. The proof of Theorem \ref{MRR1} is completed.
\end{proof}

\begin{proof}[Proof of Theorem $\ref{MRR2}$.]
To prove Theorem \ref{MRR2}, it is sufficient to construct a dense $G_\delta$ subset $F\subset [0,1]$ such that $F\subset E_{\max}^\varphi.$

We first introduce a notation. For $\omega=(\omega_1,\ldots,\omega_{m})\in \Sigma^*$ or $\omega=(\omega_1,\ldots,\omega_{m}\ldots) \in \Sigma^\infty$ and $n\in\N$ with $n\le m$, let $r_n(\omega)$ denote the length of the longest run of $1$'s in $\omega|_n=(\omega_1,\omega_2,\ldots,\omega_{n})$, that is,
\[
r_n(\omega)=\max\{\ell: \text{$\omega_{i+1}=\cdots=\omega_{i+\ell}=1$ for some $0\le i\le n-\ell$}\}.
\]

Since $\limsup\limits_{n\to \infty}\frac{n}{\varphi(n)}=+\infty$ and $\lim\limits_{n\to \infty}\varphi(n)=+\infty$, there exists a strictly monotonically increasing subsequence $\{m_k\}_{k\ge 1}\subset \mathbb{N}$ such that
\begin{equation}\label{con22}
\lim\limits_{k\to \infty}\frac{m_k}{\varphi(m_k)}=+\infty.
\end{equation}

Let $\Omega_{0} =\Sigma^*$. For each $\omega\in \Omega_{0},$ choose positive integer $n_1=n_1(\omega)\in\{m_k\}_{k\ge 1}$ such that $n_1-[\log \varphi(n_1)]- |\omega|>0$ and $[\log \varphi(n_1)]\ge |\omega|$. These conditions can always be satisfied by choosing $n_1$ large enough.

Define
\[
\Omega_1=\bigcup_{\omega\in \Omega_{0}}\omega 0^{n_1-[\log \varphi(n_1)]-|\omega|}1^{[\log \varphi(n_1)]}.
\]
For any $\tau\in \Omega_1,$ there exists some $\omega\in \Omega_{0}$ with $|\omega|=n_1(\omega)$ such that $\tau=\omega 0^{n_1-[\log \varphi(n_1)]-|\omega|}1^{[\log \varphi(n_1)]}.$ It is easy to check that $|\tau|=n_1(\omega)$ and $r_{n_1}(\tau)=[\log \varphi(n_1(\omega))]$ since $n_1-[\log \varphi(n_1)]- |\omega|>0$ and $[\log \varphi(n_1)]\ge |\omega|$.

Then, for $\tau\in \Omega_1,$ choose a positive integer $n_2=n_2(\tau)\in \{m_k\}_{k\ge 1}$ such that $n_2>2n_1$.
Define
\[
\Omega_2=\bigcup_{\tau\in \Omega_{1}}\tau 0^{n_2-[\frac{1}{2}n_2]-n_1}1^{[\frac{1}{2}n_2]}.
\]
Analogously, for any $\xi\in \Omega_2$ there exists some word $\tau\in \Omega_{1}$ with $|\tau|=n_1$ such that $\xi=\tau 0^{n_2-[\frac{1}{2}n_2]-n_1}1^{[\frac{1}{2}n_2]}$ and $|\xi|=n_2(\tau)$. It follows from $n_2> 2n_1$ that $n_2-[\frac{1}{2}n_2]-n_1>0$ and therefore $r_{n_2}(\xi)=[\frac{1}{2}n_2]$.

Suppose we have chosen the positive integers $n_1,n_2,\cdots, n_{2m-1}, n_{2m}, $ and have defined the sets $\Omega_1, \Omega_2,\cdots, \Omega_{2m-1}, \Omega_{2m},$ we next show how to define the sets $\Omega_{2m+1}$ and $\Omega_{2m+2}$.  For $\eta\in \Omega_{2m}$ with $|\eta|=n_{2m}$, choose integer $n_{2m+1}=n_{2m+1}(\eta)\in\{m_k\}_{k\ge 1}$ such that $n_{2m+1}-[\log \varphi(n_{2m+1})]-n_{2m}>0$ and $[\log \varphi(n_{2m+1})]\ge|\eta|$. Again, these conditions can always be satisfied by choosing $n_{2m+1}$ large enough.
 Define
\[
\Omega_{2m+1}=\bigcup_{\eta\in \Omega_{2m}}\eta 0^{n_{2m+1}-[\log \varphi(n_{2m+1})]-n_{2m}}1^{[\log \varphi(n_{2m+1})]}.
\]

For any $\mu\in \Omega_{2m+1},$ there exists some word $\eta\in \Omega_{2m}$ with $|\eta|=n_{2m}$ such that $\mu=\eta 0^{n_{2m+1}-[\log \varphi(n_{2m+1})]-n_{2m}}1^{[\log \varphi(n_{2m+1})]}$.  It is easy to check that $|\mu|=n_{2m+1}$. Moreover,  $r_{n_{2m+1}}(\mu)=[\log \varphi(n_{2m+1})]$ since $n_{2m+1}-[\log \varphi(n_{2m+1})]-n_{2m}>0$ and $[\log n_{2m+1}]\ge|\eta|$.

Then, for $\sigma\in\Omega_{2m+1}$ with $|\sigma|=n_{2m+1}$, choose a integer $n_{2m+2}=n_{2m+2}(\sigma)\in\{m_k\}_{k\ge 1}$ such that $n_{2m+2}>2n_{2m+1}$, and define

\[
\Omega_{2m+2}=\bigcup_{\sigma\in \Omega_{2m+1}}\sigma 0^{n_{2m+2}-[\frac{1}{2}n_{2m+2}]-n_{2m+1}}1^{[\frac{1}{2}n_{2m+2}]}.
\]
Analogously, for any $\nu\in \Omega_{2m+2},$ we have $|\nu|=n_{2m+2}$ and $r_{n_{2m+2}}(\nu)=[\frac{1}{2} n_{2m+2}]$ since $n_{2m+2}>2n_{2m+1}$.

Next, define

\[
\Omega=\bigcap_{k=0}^\infty \Omega_k.
\]
Then, we claim that $\Omega$ is residual in $\Sigma^\infty.$ In fact, $\Omega$ is a $G_\delta$ set in $\Sigma^\infty$ since it is not difficult to check that each cylinder set $\Omega_k$ is open. Moreover, by construction, each set $\Omega_k$ is dense, and so it follows from Baire's theorem that $\Omega$ is also dense in $\Sigma^\infty$.

Write $\Sigma_{\max}^\varphi=\pi^{-1}(E_{\max}^\varphi).$ we will show that
\begin{equation}\label{contain}
\Omega\subset \Sigma_{\max}^\varphi.
\end{equation}
 For $\omega\in\Omega$, it follows from the construction of the set $\Omega$ that
\[
r_{n_{2m-1}}(\omega)=[\log \varphi(n_{2m-1})], \quad  r_{n_{2m}}(\omega)=[\frac{1}{2}n_{2m}], m\ge 1.
\]
Therefore, by $\eqref{con22}$ we have
\[
\liminf\limits_{n\to \infty}\frac{r_n(\omega)}{\varphi(n)}\le\lim\limits_{m\to \infty}\frac{r_{n_{2m-1}}(\omega)}{\varphi (n_{2m-1})}=\lim\limits_{m\to \infty}\frac{[\log \varphi(n_{2m-1})]}{\varphi (n_{2m-1})}=0
 \]
 and
 \[
 \limsup\limits_{n\to \infty}\frac{r_n(\omega)}{\varphi(n)}\ge \lim\limits_{m\to \infty}\frac{r_{n_{2m}}(\omega)}{\varphi( n_{2m})}=\lim\limits_{m\to \infty}\frac{[\frac{1}{2}n_{2m}]}{\varphi( n_{2m})}=+\infty,
\]
which imply $\Omega\subset \Sigma_{\max}^\varphi$.

Finally, we use the set $\Omega$ to define the desired set $F$.

Let
\[
B=\left\{x\in [0,1]:x=\frac{k}{2^n}, k, n\in\mathbb{N}\right\},
\]
and write
\[
\widetilde{\Sigma}=\Sigma^\infty\setminus \pi^{-1}(B), \quad  \widetilde{E}=[0,1]\setminus B.
\]

The map $\pi:\widetilde{\Sigma}\to \widetilde{E}$ is bijective. Note that $\pi^{-1}(B)$ is a $F_\sigma$ set, the set $\widetilde{\Sigma}$ is a $G_\delta$ set. Moreover, it is easy to check that $\widetilde{\Sigma}$  is dense in $\Sigma^\infty.$

Now, let
\[
F=\pi(\Omega\cap \widetilde{\Sigma}).
\]
Following the argument in \cite{LW5}, we can check that $F$ is the desired set. %It suffices to check that the set $F\subset \widetilde{E}$ satisfies the following properties:

\end{proof}

% ------------------------------------------------------------------------

\subsection*{Acknowledgements}
%The authors would like to thank the anonymous referee for carefully reading our manuscript and pointing out several mistakes in a previous draft of this manuscript.
This project was supported by the National Natural Science
Foundation of China (11371148 \& 11301473), the Natural Science
Foundation of Fujian Province (2014J05008) and the Program for New Century Excellent Talents at Minnan Normal University (MX13002).

\end{document}